\documentclass[12pt]{amsart}
\usepackage{amsfonts}
\usepackage{mathrsfs}
\usepackage{amsmath,amssymb,amsthm,}
\numberwithin{equation}{section}

\def\N{\mathbb{N}{\ssc\,}}
\def\Z{\mathbb{Z}{\ssc\,}}

\def\R{\mathbb{R}{\ssc\,}}
\def\C{\mathbb{C}{\ssc\,}}

\def\ssc{\scriptscriptstyle}

\def\d{\delta}

\def \Vir{{\rm{Vir}}}

\def \span{{\rm span}_{\C}}

\def \R{{\mathcal R}}

\def \NS{{\mathcal {NS}}}

\def \T{{\mathcal T}}
\def \hh{{\mathfrak h}}
\def \gg{{\mathfrak g}}
\def \LL{{\mathcal L}}
\def\L{{\mathbb{L}}}
\def\G{{\mathbb{G}}}

\def \<{\langle}
\def \>{\rangle}

\def\vs{\vspace*}

\def \be{\begin{equation}\label}
\def \ee{\end{equation}}
\def \bex{\begin{example}\label}
\def \eex{\end{example}}
\def \bl{\begin{lem}\label}
\def \el{\end{lem}}
\def \bt{\begin{thm}\label}
\def \et{\end{thm}}
\def \bp{\begin{prop}\label}
\def \ep{\end{prop}}
\def \br{\begin{rem}\label}
\def \er{\end{rem}}
\def \bc{\begin{coro}\label}
\def \ec{\end{coro}}
\def \bd{\begin{de}\label}
\def \ed{\end{de}}

\newtheorem{thm}{Theorem}[section]
\newtheorem{prop}[thm]{Proposition}
\newtheorem{coro}[thm]{Corollary}

\newtheorem{example}[thm]{Example}
\newtheorem{lem}[thm]{Lemma}

\newtheorem{rem}[thm]{Remark}
\newtheorem{de}[thm]{Definition}

\def\adddot{$\!\!\!${\bf}\ \ }
\makeatletter \@addtoreset{equation}{section}

\makeatother \makeatletter
\begin{document}

\title[On non-weight representations of the $N=2$ superconformal algebras]{On non-weight representations of the $N=2$ superconformal algebras}
\author{Hengyun Yang}
\address{\it Hengyun Yang: Department of Mathematics, Shanghai Maritime University, Shanghai 201306, China}
\email{hyyang@shmtu.edu.cn}

\author{Yufeng Yao$^*$}
\address{\it Yufeng Yao: Department of Mathematics, Shanghai Maritime University, Shanghai 201306, China}
\email{yfyao@shmtu.edu.cn}

\author{Limeng Xia}
\address{\it Limeng Xia:  Institute of Applied System Analysis, Jiangsu University, Jiangsu Zhenjiang, 212013,
China}\email{xialimeng@ujs.edu.cn}
\thanks{$^*$The correspongding author}
\thanks{\rm This work is supported by the National Natural Science Foundation of China (Grant Nos. 11771279, 11771142, 11871249, 11801363, 11671138 and 11571008), the Natural Science Foundation of Jiangsu (Grant No. BK20171294) and the Natural Science Foundation of Shanghai (Grant No. 16ZR1415000).}

\begin{abstract} In this paper, we construct a family of non-weight modules over the untwisted $N=2$ superconformal algebras.
Those modules when regarded as modules over the Cartan subalgebra (modulo the center) are free of rank $2$. We give a classification of isomorphism classes of such modules. Moreover, all submodules of such modules are precisely determined. In particular, those modules are not simple. The corresponding  simple quotient modules are classified. Furthermore, these simple modules when restricted as modules over $N=1$ superconformal algebras coincide with those modules constructed in [H. Yang, Y. Yao, L. Xia, A family of non-weight modules over the super-Virasoro algebras, J. Algebra 547 (2020), 538-555].
\end{abstract}

\maketitle \vskip-.3cm \qquad{\small{\bf Key words:} $N=2$ superconformal algebra, Cartan subalgebra,
non-weight module}


\section{Introduction}
Superconformal algebras were first constructed by
Ademollo et al. \cite{Aet} and Kac \cite{K4} from the point of view of mathematics and physics respectively in 1970s. They are the supersymmetry extensions of the Virasoro
algebra, and play important roles in conformal field
theory and string theory. In 2002, Fattori and Kac \cite{FK} (see
also \cite{K3}) gave a complete classification of superconformal
algebras. Among those, the $N=2$ superconformal algebras play a fundamental role in the
mirror symmetry theory.

The $N=2$ superconformal algebras fall into four types: the Ramond
$N = 2$ algebra, the Neveu-Schwarz $N = 2$ algebra, the
topological $N = 2$ algebra, and the twisted $N = 2$ algebra (i.e. with mixed boundary conditions for the fermionic fields). The
first three algebras are isomorphic, and called the untwisted $N=2$ superconformal algebras. The Ramond $N = 2$ algebra
and the Neveu-Schwarz $N = 2$ algebra are isomorphic by the
spectral flow map \cite{SS}. As the symmetry algebra of
topological conformal field theory in two dimensions, the
topological $N = 2$ algebra was presented by Dijkgraaf, Verlinde
and Verlinde \cite{DVV} in 1991. This algebra can be obtained from
the Neveu-Schwarz $N=2$ algebra through modifying the stress-energy
tensor by adding the derivative of the $U(1)$ current
procedure known as ``topological twist'' (see \cite{EY, W}).

The weight representation theory of superconformal algebras
is of interest to both mathematicians and physicists. It is  known
that representations of the superconformal algebras get more
and more complicated with increasing the number of fermionic
currents $N$. The classification of all simple Harish-Chandra
modules over the $N=1$ and untwisted $N=2$ superconformal algebras
was achieved respectively in \cite{Su1, LPX3}. Moreover, some
special weight modules over the $N=2$ superconformal algebras were
studied (see \cite{FGS, IK2, LSZ, ST, YXS}). The theory of non-weight representations of superconformal algebras 
is of independent interest as a purely mathematical problem so far.

Recently, some authors constructed and studied  an important class
of non-weight modules on which the Cartan subalgebra acts freely.
These modules are called free $U(\hh)$-modules, where $U(\hh)$
is the universal enveloping algebra of the Cartan subalgebra $\hh$ (modulo the center). Free $U(\hh)$-modules were constructed by Nilsson for
the complex matrices algebra $\mathfrak{sl}_{n+1}$ in \cite{N1}. These modules over
$\mathfrak{sl}_{n+1}$ were introduced at the same time using different approach in \cite{TZ2}.
After that, such modules for finite-dimensional simple
Lie algebras \cite{N2} and some infinite-dimensional Lie algebras,
such as the Virasoro algebra \cite{CG, TZ}, the
Heisenberg-Virasoro algebra \cite{CG2}, the algebra $\Vir(a,b)$
\cite{HCS}, the Schrodinger-Virasoro algebra \cite{WZ} and the
Block algebra \cite{CY, LG}, have been studied. The paper \cite{CZ} studied free $U(\hh)$-modules over the basic
Lie superalgebras. It was shown that $\mathfrak{osp}(1|2n)$ is the
only basic Lie superalgebra that admits such modules. The $N=1$
superconformal algebras fall into two types:
the Ramond $N=1$ algebra and the Neveu-Schwarz $N=1$ algebra. In
\cite{YYX}, the authors gave a complete classification of free
$U(\hh)$-modules of rank $1$ over the Ramond $N=1$ algebra,
and free $U(\hh)$-modules of rank $2$ over the Neveu-Schwarz $N=1$
algebra. In this paper  we aim to study free $U(\hh)$-modules
over the untwisted $N=2$ superconformal algebras.

The present paper is organized as follows. In Section 2, we recall
the definition of the $N=2$ superconformal algebras and construct
a family of non-weight modules over the Ramond $N=2$
superconformal algebra. Section 3 is devoted to studying free
$U(\mathfrak{h})$-modules of rank $2$ over the Ramond $N=2$
superconformal algebra. To be precise, we classify the free
$U(\hh)$-modules of rank $2$ over the Ramond $N=2$
supeconformal algebra (Theorem \ref{thm-3}), and determine the isomorphism classes of
such modules (Theorem \ref{iso class}). Moreover, all submodules of such modules are
precisely determined (Theorem \ref{thm-sub}). In particular, those modules
are not simple. The corresponding simple quotient modules are classified (Theorem \ref{iso class of simples}). Furthermore, these simple modules
when restricted as modules over $N=1$ superconformal algebras coincide with those modules constructed in \cite{YYX} (Proposition \ref{relation}).

\section{A family of non-weight modules over $\R$}
Throughout the paper, we denote by $\C,\C^*,\Z$  and $\N$ the sets
of all complex numbers, nonzero complex numbers, integers  and
positive integers, respectively. We always assume that the base
field is the complex number field $\C$. All vector superspaces (resp.
superalgebras, supermodules) $V=V_{\bar 0}\oplus V_{\bar 1}$ are
defined over $\C$, and sometimes simply called spaces (resp.
algebras, modules). We call elements in $V_{\bar 0}$ and $V_{\bar
1}$ odd and even, respectively. Both odd and even elements are
referred to homogeneous ones. Throughout this paper, a module $M$
of a superalgebra $A$ always means a supermodule, i.e.,
$A_{\bar{i}}\cdot M_{\bar{j}}\subseteq M_{\bar{i}+\bar{j}}$ for
all $\bar{i}, \bar{j}\in\Z_2$. There is a parity change functor
$\Pi$ from the category of $A$-modules to itself. That is, for any
module $M=M_{\bar{0}}\oplus M_{\bar{1}}$, we have a new  module
$\Pi(M)$ with the same underlining space with the parity
exchanged, i.e., $(\Pi(M))_{\bar{0}}=M_{\bar{1}}$ and
$(\Pi(M))_{\bar{1}}=M_{\bar{0}}$.

In this section, we construct a family of non-weight modules over
the untwisted $N=2$ superconformal algebra of Ramond type, which are actually
free of rank 2 when restricted as modules over the Cartan subalgebra (modulo the center).

Let us first recall the definition of the untwisted  $N=2$
superconformal algebras, which include three types, that is, the
Ramond, the Neveu-Schwarz and the topological $N=2$ supeconformal
algebras.

\begin{de}\label{def-R} {\rm{(cf. \cite{Aet} )}}
Let $\LL$ be an infinite dimensional Lie superalgebra whose even
part is spanned by $\{L_{m}, H_{m}, C \mid m \in \Z \}$ and odd
part is spanned by $\{G^{\pm}_{r}\mid r \in \epsilon +\Z\}$
($\epsilon =0\, \mbox{or} \, \frac{1}{2}$) subject to the
following relations:
\begin{eqnarray*}
    &&[L_{m}, L_{n}]=(m-n)L_{m+n}+\frac{1}{12}(m^{3}-m)\delta_{m+n,
    0}C,\quad [L_{m}, H_{n}]=-nH_{m+n},\\
    &&[H_m,H_n]=\frac{1}{3}m\delta_{m+n,0}C,\quad [L_m,G^{\pm}_r]=(\frac{1}{2}m-r)G^{\pm}_{m+r},\\
    &&[H_m,G_r^{\pm}]=\pm G_{m+r}^{\pm}, \quad [G^-_r,G^+_s]=2L_{r+s}-(r-s)H_{r+s}+\frac{1}{3}(r^{2}-\frac14)\delta_{r+s,0}C,\\
    &&[G^+_r,G^+_s]=[G^-_r,G^-_s]=0,\quad [\LL,C]=0
\end{eqnarray*}
for $m,n\in\Z$, $r,s\in \epsilon+\Z$. If $\LL=\span\{L_{m},
H_{m},G^{\pm}_m, C \mid m \in \Z\}$, it is called the Ramond $N=2$
supeconformal algebra, and denoted by $\R$. If $\LL=\span\{L_{m},
H_{m},G^{\pm}_r, C \mid m \in \Z, r\in \frac12+\Z\}$, it is called
the Neveu-Schwarz $N=2$ superconformal algebra, and denoted by $\NS$.
\end{de}

Let $\sigma : \NS\rightarrow \R$ be the spectral flow map (see
\cite{SS}) defined by
\begin{equation}\label{iso between NS and R}
  L_m \mapsto L_m+\frac12 H_m+\frac{1}{24} \delta_{m,0}C, \quad H_m \mapsto   H_{m}+\frac{1}{6} \delta_{m,0}C,
  \quad
  G_{r}^{\pm}\mapsto G_{r\pm\frac12}^{\pm}, \quad C \mapsto C
\end{equation}
for $m\in\Z, r\in\frac12+\Z$. It is straightforward to show that
$\sigma$ is an isomorphism between the Neveu-Schwarz $N=2$
superconformal algebra and the Ramond $N=2$ superconformal algebra.
\begin{de}\label{def-To} {\rm{(cf.  \cite{DVV})}}
The topological $N = 2$ superconformal algebra is a Lie
superalgebra
$$\mathcal{T}=\mbox{$\bigoplus\limits_{m\in\Z}$}\C L_m \mbox{$\bigoplus\limits_{m\in\Z}$}\C H_m \mbox{$\bigoplus\limits_{m\in\Z}$}\C G_m\mbox{$\bigoplus\limits_{m\in\Z}$}\C Q_m\oplus \C
C $$ with the following brackets:
\begin{eqnarray*}
    &&[L_{m}, L_{n}]=(m-n)L_{m+n},\quad [L_{m}, H_{n}]=-nH_{m+n}+\frac16(m^2+m)\d_{m+n,0}C,\\
    &&[H_m,H_n]=\frac{1}{3}m\delta_{m+n,0}C,\quad [L_m,G_n]=(m-n)G_{m+n},\\
    && [L_m,Q_n]=-n Q_{m+n},\quad  [H_m,G_n]= G_{m+n},\\
    &&[H_m,Q_n]=-  Q_{m+n}, \quad [G_m,Q_n]=2 L_{m+n}-2nH_{m+n}+\frac{1}{3} (m^{2}+m)\delta_{m+n,0}C,\\
    &&[G_m,G_n]=[Q_m,Q_n]=0,\quad [\mathcal{T},C]=0
\end{eqnarray*}
for $m,n\in\Z$.
\end{de}
It is also known that the following map from $\T$ to $\R$ defined by
{\small{
\begin{equation*}\label{iso between T and R}
  L_m \mapsto L_m+(\frac{1}{2}m+1) H_{m}+\frac{1}{8}\delta_{m,0}C, \,
H_m \mapsto H_{m}+\frac{1}{6}\delta_{m,0}C,  \, G_{m}\mapsto G_{m+1}^{+} ,\,
Q_m\mapsto  G_{m-1}^{-} ,\, C \mapsto C
\end{equation*}}}is an  isomorphism, where  $m\in\Z$. Indeed, it is the composition of the inverse of the topological twists $\tau$ from $\NS$ to $\T$ (see
\cite{EY,W}) defined by
\begin{equation*}
L_m \mapsto L_m-\frac{1}{2}(m+1) H_{m}, \,\,
H_m \mapsto H_{m},  \,\, G_{m+\frac{1}{2}}^{+}\mapsto G_{m}  ,\,\,
 G_{m-\frac{1}{2}}^{-}\mapsto Q_m,\,\, C \mapsto C
\end{equation*}
and the map $\sigma$ defined in (\ref{iso between NS and R}). Thus the
Ramond, the Neveu-Schwarz and the topological $N=2$ superconformal
algebras are isomorphic. They are called {\it untwisted $N=2$ superconformal
algebras}. In this paper, we consider the Ramond $N=2$ superconformal
algebra as the representative of the untwisted case. More precisely, we study free $U(\hh)$-modules over the Ramond $N=2$
superconformal algebra $\R$, where $\hh=\C L_0\oplus\C H_0$ is the canonical Cartan
subalgebra (modulo the center) of $\R$.  Since $[L_0,H_0]=0$,  we
have $U(\hh)=\C[L_0,H_0]$.

Let $\C [x,y]$ and $ \C [s,t]$ be the polynomial algebras in the
indeterminates $x,y$ and $s,t$, respectively. For
$\lambda,\alpha\in\C^*$,  denote by $\Omega (\lambda,\alpha) =\C
[x,y]\oplus \C [s,t]$ the $\Z_2$-graded vector space with $\Omega
(\lambda,\alpha)_{\bar{0}}=\C [x,y]$ and $\Omega
(\lambda,\alpha)_{\bar{1}}=\C[s,t]$.  For any $m\in\Z$, $f(x,y)\in \C [x,y]$
and $g(s,t)\in \C [s,t]$, we define the action of $\R$ on  $\Omega
(\lambda,\alpha) $ as follows
    \begin{eqnarray}
    &&L_m f(x,y)=\lambda^m(x+\frac12
    my)f(x+m,y),\quad  L_m g(s,t)=\lambda^m (s+\frac
    {1}{2}mt+m)g(s+m,t),\label{module1}\\
    &&H_m f(x,y)=\lambda^m y f(x+m,y), \quad H_m g(s,t)=\lambda^m t g(s+m,t),\label{module2}\\
    &&G_m^+ f(x,y)=0 ,\quad G_m^+ g(s,t)=\lambda^m  \frac{2}{\alpha}(x+my)g(x+m,y-1) , \label{module3}\\
    &&G_m^{-} f(x,y)=\lambda^m \alpha f(s+m,t+1) , \quad G_m^{-} g(s,t)=0, \label{module4}\\
    &&Cf(x,y)=Cg(s,t)=0.\label{C action}
    \end{eqnarray}

\begin{prop} For $\lambda,\alpha\in\C^*$, $\Omega (\lambda,\alpha)$ is an $\R$-module under the action defined by
(\ref{module1})-(\ref{C action}). Moreover, $\Omega
(\lambda,\alpha)$ is free of rank 2 as a module over
$\C[L_0,H_0]$.
\end{prop}

\begin{proof} For any $m,n\in\Z$, $f(x,y)\in\C[x,y], g(s,t)\in \C [s,t]$, by (\ref{module1}), we
have
\begin{eqnarray*}
L_mL_nf(x,y)\!\!\!&=\!\!\!&\lambda^nL_m(x+\frac12
ny)f(x+n,y)\\
\!\!\!&=\!\!\!&\lambda^{m+n}(x+\frac12 my)(x+\frac12
ny+m)f(x+m+n,y),
\end{eqnarray*}
and
\begin{eqnarray*}
L_mL_ng(s,t)\!\!\!&=\!\!\!&\lambda^n L_m  (s+\frac {1}{2}nt+n)g(s+n,t)\\
\!\!\!&=\!\!\!&\lambda^{m+n}(s+\frac12 mt+m)(s+\frac
{1}{2}nt+m+n)g(s+m+n,t),
\end{eqnarray*}
which implies that
\begin{eqnarray*}
&&L_mL_nf(x,y)-L_nL_mf(x,y)\\
\!\!\!&=\!\!\!&\lambda^{m+n}\big((x+\frac12 my)(x+\frac12
ny+m)-(x+\frac12 ny)(x+\frac12 my+n) \big)f(x+m+n,y)\\
\!\!\!&=\!\!\!&\lambda^{m+n}(m-n)\big((x+\frac12 (m+n)y)
\big)f(x+m+n,y)\\
\!\!\!&=\!\!\!&((m-n)L_{m+n}+\frac{1}{12}(m^3-m)\delta_{m+n,0}C)f(x,y),
\end{eqnarray*}
and
\begin{eqnarray*}
&&L_mL_ng(s,t)-L_nL_mg(s,t)\\
\!\!\!&=\!\!\!&\lambda^{m+n}\big((s+\frac12 mt+m)(s+\frac
{1}{2}nt+m+n)\\
\!\!\!& \!\!\!&-(s+\frac12 nt+n)(s+\frac
{1}{2}mt+m+n)\big)g(s+m+n,t)\\
\!\!\!&= \!\!\!&\lambda^{m+n}(m-n)\big((s+\frac12
(m+n)t+m+n)\big)g(s+m+n,t)\\
\!\!\!&=\!\!\!&((m-n)L_{m+n}+\frac{1}{12}(m^3-m)\delta_{m+n,0}C)g(s,t).
\end{eqnarray*}
If we take into account (\ref{module1}) and (\ref{module2}), it
follows that
\begin{eqnarray*}
&& L_mH_nf(x,y)-H_nL_mf(x,y)\\
\!\!\!&=\!\!\!&\lambda^nL_m y f(x+n,y)-\lambda^mH_n (x+\frac12
my)f(x+m,y)\\
\!\!\!&=\!\!\!&\lambda^{m+n}\big((x+\frac12 my) -(x+\frac12
my+n)\big)y f(x+m+n,y)\\
\!\!\!&=\!\!\!&-n\lambda^{m+n}y f(x+m+n,y)\\
\!\!\!&=\!\!\!&-n H_{m+n}f(x,y).
\end{eqnarray*}
Similarly,  it is easy to prove that
\begin{eqnarray*}
 L_mH_ng(s,t)-H_nL_mg(s,t)=-nH_{m+n}g(s,t).
\end{eqnarray*}
It follows from (\ref{module1}) and (\ref{module3}) that
\begin{eqnarray*}
L_mG_n^+g(s,t)\!\!\!&=\!\!\!&\frac{2}{\alpha}\lambda^n   L_m(x+ny)g(x+n,y-1)\\
\!\!\!&=\!\!\!&\frac{2}{\alpha}\lambda^{m+n}(x+\frac12 my)(x+
ny+m)g(x+m+n,y-1),
\end{eqnarray*}
and
\begin{eqnarray*}
G_n^+L_mg(s,t)\!\!\!&=\!\!\!&\lambda^m G_n^+(s+\frac {1}{2}mt+m)g(s+m,t)\\
\!\!\!&=\!\!\!&\frac{2}{\alpha}\lambda^{m+n}(x+ ny)(x+\frac12
my+\frac12m+n)g(x+m+n,y-1).
\end{eqnarray*}
Hence,
\begin{eqnarray*}
&& L_mG_n^+g(s,t)- G_n^+L_mg(s,t)\\
\!\!\!&=\!\!\!&\frac{2}{\alpha}\lambda^{m+n}\big((x+\frac12
my)(x+ny+m)\\
\!\!\!& \!\!\!&-(x+ ny)(x+\frac12 my+\frac12m+n)
\big)g(x+m+n,y-1)\\
\!\!\!&= \!\!\!&(\frac12 m-n)\frac{2}{\alpha}\lambda^{m+n}  (x+(m+n)y)g(x+m+n,y-1) \\
\!\!\!&=\!\!\!&(\frac12 m-n)G_{m+n}^+g(s,t).
\end{eqnarray*}

According to (\ref{module1}), (\ref{module2}) and (\ref{module3}), we obtain
\begin{eqnarray*}
&& L_mG_n^-f(x,y)- G_n^-L_mf(x,y)\\
\!\!\!&=\!\!\!& \alpha \lambda^{n}L_mf(s+n,t+1)-\lambda^m
G_n^-(x+\frac12 my)f(x+m,y)\\
\!\!\!&= \!\!\!&\alpha\lambda^{m+n} \big(
s+\frac12mt+m-(s+n+\frac12m(t+1))\big)f(s+m+n,t+1)\\
\!\!\!&=\!\!\!&(\frac12 m-n)\alpha\lambda^{m+n}f(s+m+n,t+1)\\
\!\!\!&=\!\!\!&(\frac12 m-n)G_{m+n}^-f(x,y),
\end{eqnarray*}
and
\begin{eqnarray*}
&& H_mG_n^+g(s,t)-G^+_nH_mg(s,t)\\
\!\!\!&=\!\!\!& \frac{2}{\alpha}\lambda^n
H_m(x+ny)g(x+n,y-1)-\lambda^mG_n^+t g(s+m,t)\\
\!\!\!&=\!\!\!& \frac{2}{\alpha}\lambda^{m+n}\big(
(x+ny+m)y-(x+ny)(y-1)\big)g(x+m+n,y-1)\\
\!\!\!&=\!\!\!& \frac{2}{\alpha}\lambda^{m+n}\big(
x+(m+n)y\big)g(x+m+n,y-1)\\
\!\!\!&=\!\!\!&  G^+_{m+n}g(s,t).
\end{eqnarray*}
In a similar way one sees that
\begin{eqnarray*}
[H_m,G_n^-]f(x,y)=-G_{m+n}^-f(x,y).
\end{eqnarray*}

By (\ref{module1})-(\ref{C action}), we have
\begin{eqnarray*}
&&G^-_m G^+_n f(x,y)+G^+_nG^-_m f(x,y)\\
\!\!\!&=\!\!\!& \alpha\lambda^mG^+_n  f(s+m,t+1)\\
\!\!\!&=\!\!\!& 2\lambda^{m+n}(x+ny)  f(x+m+n,y),
\end{eqnarray*}
and
\begin{eqnarray*}
&&2L_{m+n}f(x,y)-(m-n)H_{m+n}f(x,y)+\frac{1}{3}(m^2-\frac{1}{4})\delta_{m+n,0}Cf(x,y)\\
\!\!\!&=\!\!\!&\lambda^{m+n}\big(2(x+\frac12(m+n)y)-(m-n)y
  \big)f(x+m+n,y)\\
\!\!\!&=\!\!\!& 2\lambda^{m+n}(x+ny)  f(x+m+n,y)  .
\end{eqnarray*}
That is
$$[G^-_m,G^+_n]f(x,y)=(2L_{m+n}-(m-n)H_{m+n}+\frac{1}{3}(m^2-\frac{1}{4})\delta_{m+n,0}C)f(x,y).$$

A similar computation yields that
\begin{eqnarray*}
[G^-_m,G^+_n]g(s,t)=(2L_{m+n}-(m-n)H_{m+n}+\frac{1}{3}(m^2-\frac{1}{4})\delta_{m+n,0}C)g(s,t).
\end{eqnarray*}
Moreover, it is easy to show that
\begin{eqnarray*}
[L_m,G^+_n]f(x,y)=(\frac12 m-n)G_{m+n}^+f(x,y)=0,\quad
[H_m,G^+_n]f(x,y)=G^+_{m+n}f(x,y)=0;
\end{eqnarray*}
\begin{eqnarray*}
[H_m,G_n^-]g(s,t)=-G_{m+n}^-g(s,t)=0,\quad [H_m, H_n]f(x,y)=[H_m,H_n]g(s,t)=0;
\end{eqnarray*}
\begin{eqnarray*}
[L_m,G_n^-]g(s,t)=(\frac{1}{2}m-n)G_{m+n}^-g(s,t)=0;
\end{eqnarray*}

\begin{eqnarray*}
[G_m^+,G_n^+]f(x,y)=[G_m^-,G_n^-]f(x,y)=[G_m^+,G_n^+]g(s,t)=[G_m^-,G_n^-]g(s,t)=0;
\end{eqnarray*}
and

$$[\R, C]f(x,y)=[\R, C]g(s,t)=0.$$\

In conclusion, we have shown that $\Omega(\lambda,\alpha)$ is an $\R$-module. By taking $m=0$ in (\ref{module1}) and (\ref{module2}), we see that $\Omega
(\lambda,\alpha)$ is free of rank 2 as a module over $\C[L_0,H_0]$. We complete the proof.
\end{proof}

\section{Free $U(\mathfrak{h})$-modules of rank $2$ over $\R$}
\subsection{Some key lemmas and consequences} To give the classification of free $U(\mathfrak{h})$-modules of rank $2$ over $\R$, in this subsection, we do some preparation and give some key lemmas and their consequences. We first present the following easy observation, the proof of which is straightforward.
\begin{lem}\label{easy lem}
Let $\gg=\gg_{\bar{0}}\oplus\gg_{\bar{1}}$ be a Lie superalgebra with a subalgebra $\hh\subseteq\gg_{\bar{0}}$, and $[\gg_{\bar{1}},\gg_{\bar{1}}]=\gg_{\bar{0}}$. Then there do not exist $\gg$-modules which are free of rank 1 as $U(\hh)$-modules.
\end{lem}

The Ramond $N=2$ supeconformal algebra $\R$ has a  canonical Cartan
subalgebra (modulo the center) $\hh=\C L_0\oplus\C H_0\subseteq\R_{\bar 0}$. Moreover, $\R$ is generated
by odd elements $G_m^{\pm}, m\in\Z$. Hence, by Lemma \ref{easy lem},
 there do not exist $\R$-modules which are free over
$U(\mathfrak{h})$ of rank $1$.

In this section, we classify the free
$U(\mathfrak{h})$-modules of rank $2$ over the Ramond $N=2$
superconformal algebra $\R$. Moreover, we determine the isomorphism classes of such modules and we precisely give all submodules of such modules. In particular, the free $U(\mathfrak{h})$-modules of rank $2$ over the Ramond $N=2$
superconformal algebra $\R$ are not simple.

Let $M=M_{\bar 0}\oplus M_{\bar 1}$ be an $\R$-module such that it
is free of rank $2$ as a $U(\mathfrak{h})$-module with two
homogeneous basis elements $v$ and $w$. Obviously, $v$ and $w$
have different parities. Set $v={\bf 1}_{\bar 0}\in M_{\bar 0}$
and $w={\bf 1}_{\bar 1}\in M_{\bar 1}$. We may assume
$$M=U(\hh){\bf 1}_{\bar 0}\oplus U(\hh){\bf 1}_{\bar
1}=\C[L_0,H_0]{\bf 1}_{\bar 0}\oplus\C[L_0,H_0]{\bf 1}_{\bar 1}$$
with $M_{\bar 0}=\C[L_0,H_0]{\bf 1}_{\bar 0}$ and $M_{\bar
1}=\C[L_0,H_0]{\bf 1}_{\bar 1}$.

We need the following preliminary result for later use.

\begin{lem}\label{relations of R}\adddot For any $m\in\Z$, $n\in\N$, we
have

{\em (1)} $L_mL_0^n=(L_0+m)^nL_m$, $L_mH_0^n=H_0^nL_m$.

{\em (2)} $H_mL_0^n=(L_0+m)^nH_m$, $H_mH_0^n=H_0^nH_m$.

{\em (3)} $G_m^{\pm}L_0^n=(L_0+m)^nG_m^{\pm}$,
$G_m^{\pm}H_0^n=(H_0\mp 1)^nG_m^{\pm}$.
\end{lem}

\begin{proof}
According to Definition \ref{def-R}, it is easy to check that
\begin{eqnarray*}
&&L_mL_0=(L_0+m)L_m,\quad L_mH_0 =H_0 L_m,\quad H_mL_0 =(L_0+m)
H_m,\\
&& H_mH_0 =H_0 H_m, \quad  G_m^{\pm}L_0 =(L_0+m) G_m^{\pm},\quad
G_m^{\pm}H_0 =(H_0\mp 1) G_m^{\pm}.
\end{eqnarray*}
Then the lemma can be proven by induction on $n$.
\end{proof}

The following assertion on the action of $H_m\, (m\in\Z)$ is crucial for our further discussion.

\begin{lem}\label{lem1' for
 R}\adddot
For any $ m\in\Z$, we have $H_m{\bf 1}_{\bar 0}=c_m(H_0){\bf 1}_{\bar
0}, H_m{\bf 1}_{\bar 1}=c'_m(H_0){\bf 1}_{\bar 1}$ with
$c_m(H_0),c'_m(H_0)\in\C[H_0]$.
\end{lem}
\begin{proof}
Suppose
\begin{eqnarray*}
    H_m{\bf 1}_{\bar 0}=\sum_{i=0}^{k_m}c_{m,i}(H_0)L_0^i{\bf 1}_{\bar 0},\label{Hm10}
\end{eqnarray*}
where $c_{m,i}(H_0)\in\C[H_0], k_m\in\mathbb{N}$ and $c_{m, k_m}(H_0)\neq 0$. Having
in mind $[H_m,H_n]=0$ for $m+n\neq 0$ and by applying Lemma
\ref{relations of R} (2), we have
\begin{eqnarray*}
 0\!\!\!&=&\!\!\!H_mH_n{\bf 1}_{\bar 0}-H_nH_m{\bf 1}_{\bar
0}\\
\!\!\!&=&\!\!\!\sum_{i=0}^{k_n}c_{n,i}(H_0)(L_0+m)^iH_m{\bf 1}_{\bar
0}-\sum_{i=0}^{k_m}c_{m,i}(H_0)(L_0+n)^iH_n{\bf
1}_{\bar 0}\\
\!\!\!&=&\!\!\! c_{m,k_m}(H_0)c_{n,k_n}(H_0)(L_0+m)^{k_n}L_0^{k_m}{\bf
1}_{\bar 0}-c_{m,k_m}(H_0)c_{n,k_n}(H_0)(L_0+n)^{k_m}L_0^{k_n}{\bf
1}_{\bar 0} \\
&&\!\!\!+\sum_{i=0}^{k_n-1}c_{n,i}(H_0)(L_0+m)^iH_m{\bf 1}_{\bar
0}-\sum_{i=0}^{k_m-1}c_{m,i}(H_0)(L_0+n)^iH_n{\bf
1}_{\bar 0}\\
\!\!\!&=&\!\!\! c_{m,k_m}(H_0)c_{n,k_n}(H_0)(nk_m-mk_n) L_0^{k_m+k_n-1}{\bf
1}_{\bar 0}+ \mbox{ lower degree terms with respect to}\, L_0^i{\bf 1}_{\bar
0}.
\end{eqnarray*}
Since $c_{m,k_m}(H_0)c_{n,k_n}(H_0)\neq 0$, we obtain $nk_m-mk_n=0$ for all
$m,n\in\Z$  with $m+n\neq 0$. This implies $k_m=0$ for all $m\in\Z$. Thus $H_m{\bf
1}_{\bar 0}\in\C[H_0]{\bf 1}_{\bar 0}$. Similar arguments yield $H_m{\bf 1}_{\bar 1}\in\C[H_0]{\bf 1}_{\bar 1}$ for all
$m\in\Z$.
\end{proof}

The following assertion on the trivial action of the central element follows directly from Lemma \ref{lem1' for
 R}.
\begin{coro}\label{triviality}
The central element $C$ acts on $M$ trivially.
\end{coro}

\begin{proof}
It suffices to show that $C{\bf 1}_{\bar 0}=C{\bf 1}_{\bar 1}=0$. For that, we note that $[H_1, H_{-1}]=\frac{1}{3}C$. Then it follows from
Lemma \ref{lem1' for R} that
\begin{eqnarray*}
C{\bf 1}_{\bar 0}\!\!\!&=&\!\!\!3H_1H_{-1}{\bf 1}_{\bar 0}-3H_{-1}H_1{\bf 1}_{\bar
0}\\
\!\!\!&=&\!\!\!3H_1c_{-1}(H_0){\bf 1}_{\bar
0}-3H_{-1}c_{1}(H_0){\bf 1}_{\bar 0}\\
\!\!\!&=&\!\!\!3c_{-1}(H_0)H_1{\bf 1}_{\bar
0}-3c_{1}(H_0)H_{-1}{\bf 1}_{\bar 0}\\
\!\!\!&=&\!\!\! 3c_{-1}(H_0)c_{1}(H_0){\bf 1}_{\bar
0}-3c_{1}(H_0)c_{-1}(H_0){\bf 1}_{\bar 0}\\
\!\!\!&=&\!\!\! 0.
\end{eqnarray*}
Similar argument yields that $C{\bf 1}_{\bar 1}=0$. We complete the proof.
\end{proof}

\begin{lem}\label{lem1 for R}\adddot
For any $m\in\Z$, we have $L_m{\bf 1}_{\bar 0}\neq 0$, $L_m{\bf
1}_{\bar 1}\neq 0$, $H_m{\bf 1}_{\bar 0}\neq 0$, $H_m{\bf 1}_{\bar
1}\neq 0$.
\end{lem}
\begin{proof}
Suppose on the contrary that $L_{m_0}{\bf 1}_{\bar 0}=0$ and
$H_{n_0}{\bf 1}_{\bar 0}=0$ for some nonzero integer $m_0$ and
$n_0$. By Lemma \ref{relations of R} (1) and Lemma \ref{relations of R} (2), for any
$f(L_0,H_0){\bf 1}_{\bar 0}\in M_{\bar 0}$, we know that
\begin{eqnarray*}
&& L_{m_0}f(L_0,H_0){\bf 1}_{\bar 0}=f(L_0+m_0,H_0)L_{m_0}{\bf
1}_{\bar 0}=0,\\
&& H_{n_0}f(L_0,H_0){\bf 1}_{\bar 0}=f(L_0+n_0,H_0)H_{n_0}{\bf
1}_{\bar 0}=0.
\end{eqnarray*}
Since $[L_{m_0},L_{-m_0}]{\bf 1}_{\bar
0}=2m_{0}L_0{\bf 1}_{\bar 0}$ and $[L_{-n_0},H_{n_0}]{\bf 1}_{\bar
0}=-n_{0}H_0{\bf 1}_{\bar 0}$, it follows that
\begin{eqnarray*}
&& 2m_{0}L_0{\bf 1}_{\bar 0}=L_{m_0}L_{-m_0}{\bf 1}_{\bar
0}-L_{-m_0}L_{m_0}{\bf 1}_{\bar 0}=0,\\
&& -n_{0}H_0{\bf 1}_{\bar 0}=L_{-n_0}H_{n_0}{\bf 1}_{\bar
0}-H_{n_0}L_{-n_0}{\bf 1}_{\bar 0}=0,
\end{eqnarray*}
which is a contradiction. Similarly, we can prove $L_m{\bf
1}_{\bar 1}\neq 0$ and  $H_m{\bf 1}_{\bar 1}\neq 0$ for all
$m\in\Z$.
\end{proof}

\begin{lem}\label{lem2 for R}\adddot
For any $m\in\Z$,  the following one and only one case happens.

{\em (1)} $G_m^+{\bf 1}_{\bar 0}=G_m^-{\bf 1}_{\bar 1}=0, G_m^+
{\bf 1}_{\bar 1}\neq 0, G_m^- {\bf 1}_{\bar 0}\neq 0.$

{\em (2)} $G_m^+{\bf 1}_{\bar 1}=G_m^-{\bf 1}_{\bar 0}=0, G_m^+
{\bf 1}_{\bar 0}\neq 0, G_m^- {\bf 1}_{\bar 1}\neq 0.$

\end{lem}
\begin{proof}
 Assume that
\begin{eqnarray}
    &G_m^+{\bf 1}_{\bar 0}=f_m^+(L_0,H_0){\bf 1}_{\bar 1}, & G_m^-{\bf 1}_{\bar 0}=f_m^-(L_0,H_0){\bf 1}_{\bar
    1},\label{Gm01}\\
    &G_m^+{\bf 1}_{\bar 1}=g_m^+(L_0,H_0){\bf 1}_{\bar 0}, & G_m^-{\bf 1}_{\bar 1}=g_m^-(L_0,H_0){\bf 1}_{\bar
    0} ,\label{Gm10}
\end{eqnarray}
where  $f_m^{\pm}(L_0,H_0),g_m^{\pm}(L_0,H_0)\in\C [L_0,H_0]$.
It follows from  Lemma \ref{relations of R} (3) and $[G_0^+,G_0^+]=0$ that
\begin{equation*}\aligned
0&=(G_0^+)^2{\bf 1}_{\bar 0}\\
&=G_0^+ f_0^+(L_0,H_0){\bf 1}_{\bar
1}\\
&= f_0^+(L_0,H_0-1)G_0^+ {\bf 1}_{\bar
1}\\
&= f_0^+(L_0,H_0-1)g_0^+(L_0,H_0) {\bf 1}_{\bar 0}.
\endaligned
\end{equation*}
Then we have either $f_0^+(L_0,H_0-1)=0$ or $g_0^+(L_0,H_0)=0$.
Hence, $f_0^+(L_0,H_0)=0$ or $g_0^+(L_0,H_0)=0$,
i.e.,
\begin{equation*}\label{G0+1 for R}
\mbox{either \ \ } G_0^+{\bf 1}_{\bar 0}=0  \mbox{\ \  or\ \ }
G_0^+{\bf 1}_{\bar 1}=0.
\end{equation*}
By similar discussion, we have
\begin{equation}\label{G0-1 for R}
\mbox{either \ \ } G_0^-{\bf 1}_{\bar 0}=0  \mbox{\ \ or\ \ }
G_0^-{\bf 1}_{\bar 1}=0.
\end{equation}

{\it Case 1: } $G_0^+{\bf 1}_{\bar 0}=0$.

From $[G_0^-,G_0^+]=2L_0-\frac{1}{12}C$, one sees that
$$G_0^+G_0^-{\bf 1}_{\bar
0}=G_0^- G_0^+{\bf 1}_{\bar 0}+G_0^+G_0^-{\bf 1}_{\bar 0}=2L_0{\bf
1}_{\bar 0}\neq0.$$ Thus $G_0^-{\bf 1}_{\bar 0}\neq 0$. This
together with (\ref{G0-1 for R}) forces $G_0^-{\bf 1}_{\bar 1}=
0$. Using this result and the relation
$$G_0^-G_0^+{\bf 1}_{\bar
1}=G_0^- G_0^+{\bf 1}_{\bar 1}+G_0^+G_0^-{\bf 1}_{\bar 1}=2L_0{\bf
1}_{\bar 1}\neq0,$$ we further obtain that $G_0^+{\bf 1}_{\bar 1}\neq 0$. For any
$m\in\Z$, since $[G_0^{\pm},G_m^{\pm}]=0$, we have
$$0=G_0^+ G_m^+{\bf 1}_{\bar 0}+G_m^+G_0^+{\bf 1}_{\bar 0}=G_0^+ f_m^+(L_0,H_0){\bf 1}_{\bar 1}=f_m^+(L_0,H_0-1)G_0^+ {\bf 1}_{\bar 1},$$
and
$$0=G_0^- G_m^-{\bf 1}_{\bar 1}+G_m^-G_0^-{\bf 1}_{\bar
1}=G_0^- g_m^-(L_0,H_0){\bf 1}_{\bar 0}=g_m^-(L_0,H_0+1)G_0^- {\bf
1}_{\bar 0}.$$ These force $f_m^+(L_0,H_0-1)=0$ and
$g_m^-(L_0,H_0+1)=0$. Thus we have  $f_m^+(L_0,H_0)=0$ and
$g_m^-(L_0,H_0)=0$, i.e.,
\begin{equation*}\label{Gm+- for R}
G_m^+{\bf 1}_{\bar 0}=0, \ \ \ \ G_m^-{\bf 1}_{\bar
1}=0,\,\,\forall\,m\in\Z.
\end{equation*}
This together with $[G_m^-,G_m^+]=2L_{2m}+\frac{1}{3}(m^2-\frac{1}{4})\delta_{m,0}C$, Corollary \ref{triviality} and Lemma \ref{lem1 for
R} yield that
\begin{eqnarray*}
&& G_m^+G_m^-{\bf 1}_{\bar 0}=G_m^-G_m^+{\bf 1}_{\bar 0}+G_m^+G_m^-{\bf 1}_{\bar 0}=2L_{2m}{\bf 1}_{\bar 0}\neq 0, \\
&& G_m^-G_m^+{\bf 1}_{\bar 1}=G_m^-G_m^+{\bf 1}_{\bar
1}+G_m^+G_m^-{\bf 1}_{\bar 1}=2L_{2m}{\bf 1}_{\bar 1}\neq
0,\,\,\forall\,m\in\Z.
\end{eqnarray*}
Thus
\begin{equation*}\label{Gm+- for R}
G_m^- {\bf 1}_{\bar 0}\neq 0, \ \ \ \ G_m^+ {\bf 1}_{\bar 1}\neq
0,\,\,\forall\,m\in\Z.
\end{equation*}
Now part (1) holds in this case.

 {\it Case 2: } $G_0^+{\bf 1}_{\bar 1}=0$.

In this case, by using similar arguments, we can prove part (2).
\end{proof}

By exchanging the parity, we know that a module $M$ satisfying
Lemma \ref{lem2 for
 R} (1) is isomorphic to a module $M$ satisfying
Lemma \ref{lem2 for
 R} (2). Thus in the following, we always assume

 \begin{equation}\label{reduction}
 G_m^+{\bf 1}_{\bar 0}=G_m^-{\bf 1}_{\bar 1}=0, G_m^+
{\bf 1}_{\bar 1}\neq 0, G_m^- {\bf 1}_{\bar 0}\neq 0\,\,\text{for\, all}\,\,m\in\Z.
\end{equation}

\begin{lem}\label{lem3 for
 R}\adddot
Up to a parity, for any $ m\in\Z$,  we have $G_m^- {\bf 1}_{\bar
0}=a_{-m}{\bf 1}_{\bar 1}, G_m^+ {\bf 1}_{\bar
1}=\frac{2}{a_{m}}(L_0+mH_0){\bf 1}_{\bar 0}$, where $0\neq
a_m\in\C$.
\end{lem}
\begin{proof}
Note that $[G_m^-,G_{-m}^+]{\bf 1}_{\bar 0}=2L_0{\bf 1}_{\bar
0}-2mH_0{\bf 1}_{\bar 0}$ for $m\in\Z$. By (\ref{Gm01}),
(\ref{Gm10}), (\ref{reduction}) and Lemma \ref{relations of R} (3), we have
\begin{equation*}\aligned
&G_m^-G_{-m}^+ {\bf 1}_{\bar 0}+G_{-m}^+G_m^-{\bf 1}_{\bar
0}\\
=&G_{-m}^+f_m^-(L_0,H_0){\bf 1}_{\bar
1}\\
=&f_m^-(L_0-m,H_0-1)G_{-m}^+{\bf 1}_{\bar 1}\\
=& f_m^-(L_0-m,H_0-1)g_{-m}^+(L_0,H_0){\bf 1}_{\bar 0}.
\endaligned
\end{equation*}
Therefore,
\begin{equation*}\label{lem3-eq1}
f_m^-(L_0-m,H_0-1)g_{-m}^+(L_0,H_0)=2L_0-2mH_0.
\end{equation*}
It follows that
\begin{eqnarray*}
f_m^-(L_0-m,H_0-1)=a_{-m}, \ \ \ \
g_{-m}^+(L_0,H_0)=\frac{2}{a_{-m}}L_0-\frac{2m}{a_{-m}}H_0,
\end{eqnarray*}
or
\begin{eqnarray*}
f_m^-(L_0-m,H_0-1)=\frac{2}{b_{m}}L_0-\frac{2m}{b_{m}}H_0, \ \ \ \
g_{-m}^+(L_0,H_0)=b_{m},
\end{eqnarray*}
where $a_m, b_m\in\C^*$. Next we prove the following claim.

\textbf{Claim:} only one of the following two cases can happen.
\begin{itemize}
\item[(i)] $G_m^- {\bf 1}_{\bar 0}=a_{-m}{\bf 1}_{\bar 1}, G_{-m}^+
{\bf 1}_{\bar 1}=\frac{2}{a_{-m}}(L_0-mH_0){\bf 1}_{\bar 0}$ for
all $m \in\Z$, $a_m\in\C^*$.
\item[(ii)] $G_m^- {\bf 1}_{\bar 0}=\frac{2}{b_{m}}(L_0-mH_0){\bf
1}_{\bar 1},
 G_{-m}^+ {\bf 1}_{\bar
1}=b_{m}{\bf 1}_{\bar 0}$ for all $m \in\Z$, $b_m\in\C^*$.
\end{itemize}
Indeed, by the above discussion, without loss of generality, we can
assume that $G_{n}^- {\bf
1}_{\bar 0}=a_{-n}{\bf 1}_{\bar 1}, G_{-n}^+ {\bf 1}_{\bar
1}=\frac{2}{a_{-n}}(L_0-nH_0){\bf 1}_{\bar 0}$ for some $n\in\Z$. If there exists some
$m\in\Z$ such that $G_{m+n}^- {\bf 1}_{\bar
0}=\frac{2}{b_{m+n}}(L_0-(m+n)H_0){\bf 1}_{\bar 1},
 G_{-m-n}^+ {\bf 1}_{\bar
1}=b_{m+n}{\bf 1}_{\bar 0}$. Then, on one hand, it follows from Lemma \ref{relations of R} (3) and Lemma \ref{lem1' for
 R} that
\begin{eqnarray*}
\!\!\!&&\!\!\![H_{m}, G_{n}^-]{\bf 1}_{\bar 0}\\
\!\!\!&=&\!\!\!H_{m}G_{n}^-{\bf 1}_{\bar 0}-G_{n}^-H_{m}{\bf 1}_{\bar 0}\\
\!\!\!&=&\!\!\! a_{-n}H_{m}{\bf 1}_{\bar 1}-c_{m}(H_0)G_{n}^-{\bf 1}_{\bar 0}\\
\!\!\!&=&\!\!\! a_{-n}\big(c'_{m}(H_0)-c_{m}(H_0)\big){\bf 1}_{\bar 1}\in\C[H_0]{\bf 1}_{\bar 1}.
\end{eqnarray*}
On the other hand,
$$[H_{m}, G_{n}^-]{\bf 1}_{\bar 0}=-G_{m+n}^-{\bf 1}_{\bar 0}=-\frac{2}{b_{m+n}}(L_0-(m+n)H_0){\bf 1}_{\bar 1}\notin\C[H_0]{\bf 1}_{\bar 1},$$
which is a contradiction. Hence, the claim follows, and we complete the proof.
\end{proof}

\begin{lem}\label{lem4 for
 R}\adddot
Up to a parity,  there exist $\lambda,\alpha\in\C^*$ such that for any $m\in\Z$, we have
\begin{eqnarray}
    &G_m^- {\bf 1}_{\bar 0}=\lambda^{m}\alpha{\bf 1}_{\bar 1},& G_m^+ {\bf
1}_{\bar 1}=\frac{2}{\alpha}\lambda^{m}(L_0+mH_0){\bf 1}_{\bar 0},\label{Gm+-}\\
& H_m{\bf 1}_{\bar 0}=\lambda^m H_0{\bf 1}_{\bar 0},& H_m{\bf
1}_{\bar 1}=\lambda^mH_0{\bf 1}_{\bar 1}\label{Hm01}.
\end{eqnarray}
\end{lem}
\begin{proof}
For $ m,n\in\Z$, it follows from Lemma \ref{relations of R} (2),  Lemma \ref{relations of R} (3), Lemma  \ref{lem1' for R} and Lemma \ref{lem3 for R} that
\begin{eqnarray*}
\!\!\!&&\!\!\!H_mG^+_n{\bf 1}_{\bar 1}-G^+_nH_m{\bf 1}_{\bar
1}\\
\!\!\!&=&\!\!\!\frac{2}{a_{n}}H_m(L_0+nH_0){\bf 1}_{\bar
0}-G^+_nc'_m(H_0){\bf
1}_{\bar 1}\\
\!\!\!&=&\!\!\!\frac{2}{a_{n}}(L_0+m+nH_0)H_m{\bf 1}_{\bar
0}-c'_m(H_0-1)G^+_n{\bf
1}_{\bar 1}\\
\!\!\!&=&\!\!\!
\frac{2}{a_{n}}\big((L_0+m+nH_0)c_m(H_0)-c'_m(H_0-1)(L_0+nH_0)\big){\bf
1}_{\bar 0}\\
\!\!\!&=&\!\!\!
\frac{2}{a_{n}}\big(mc_m(H_0)+(c_m(H_0)-c'_m(H_0-1))L_0+(c_m(H_0)-c'_m(H_0-1))nH_0\big){\bf
1}_{\bar 0},
\end{eqnarray*}
where $a_{n}\in\C^*, c_m(H_0),c'_m(H_0)\in\C[H_0]$. Note that
\begin{eqnarray}\label{HG+}
[H_m,G_n^+]{\bf 1}_{\bar 1}=G_{m+n}^+{\bf 1}_{\bar
1}=\frac{2}{a_{m+n}}(L_0+(m+n)H_0){\bf 1}_{\bar 0}.
\end{eqnarray}
Comparing the coefficients of $L_0$ in (\ref{HG+}), we obtain
\begin{eqnarray}\label{HG+1}
c_m(H_0)-c'_m(H_0-1)=\frac{a_{n}}{a_{m+n}}, \,\,\forall\,m,n\in\Z.
\end{eqnarray}
Using this result and comparing the coefficients of $H_0$ in
(\ref{HG+}), we further obtain
\begin{eqnarray*}
\frac{2}{a_{n}}\big(mc_m(H_0)+\frac{a_{n}}{a_{m+n}}nH_0\big)=\frac{2}{a_{m+n}}(m+n)H_0,
\,\,\forall\,m,n\in\Z.
\end{eqnarray*}
Then we get
\begin{eqnarray*}\label{bm}
c_m(H_0)=\frac{a_{n}}{a_{m+n}}H_0, \,\,\forall\,m,n\in\Z.
\end{eqnarray*}
Thus there exits nonzero $\lambda,\alpha\in\C$ such that
\begin{eqnarray*}
a_m=\lambda^{-m}\alpha,\ \ \ \ c_m(H_0)=\lambda^mH_0,
\,\,\forall\,m\in\Z.
\end{eqnarray*}
Moreover, it follows from (\ref{HG+1}) that
\begin{eqnarray*}\label{bm'}
 c'_m(H_0)=\lambda^mH_0,
\,\,\forall\,m\in\Z.
\end{eqnarray*}
This implies (\ref{Gm+-}) and (\ref{Hm01}). The
proof is done.
\end{proof}

\begin{lem}\label{lem5 for
 R}\adddot
Up to a parity,  for any $ m\in\Z$, we have
\begin{eqnarray}\label{Lm01}
    L_m {\bf 1}_{\bar 0}=\lambda^m( L_0+\frac12 mH_0){\bf 1}_{\bar 0}, \ \ \ \ L_m {\bf
1}_{\bar 1}=\lambda^{m}(L_0+\frac12 mH_0+m){\bf 1}_{\bar 1}.
\end{eqnarray}
\end{lem}
\begin{proof}
Since
\begin{eqnarray*}\label{GmG0}
[G_m^-,G_0^+]{\bf 1}_{\bar 0}=2L_m{\bf 1}_{\bar 0}-mH_m{\bf
1}_{\bar 0},\,\,\,[G_m^-,G_0^+]{\bf 1}_{\bar 1}=2L_m{\bf 1}_{\bar 1}-mH_m{\bf
1}_{\bar 1},
\end{eqnarray*}
it follows from (\ref{reduction}), (\ref{Gm+-}) and (\ref{Hm01}) that
\begin{eqnarray*}
    L_m{\bf 1}_{\bar 0}\!\!\!&=\!\!\!&\frac12(G_m^-G_0^+{\bf 1}_{\bar 0}+G_0^+G_m^-{\bf 1}_{\bar 0}+mH_m{\bf 1}_{\bar 0})\\
    \!\!\!&=\!\!\!&\frac12( \lambda^m\alpha G_0^+{\bf 1}_{\bar 1}+m\lambda^mH_0{\bf 1}_{\bar 0})\\
    \!\!\!&=\!\!\!& \lambda^m( L_0+\frac12 mH_0){\bf 1}_{\bar 0},
\end{eqnarray*}
and
\begin{eqnarray*}
    L_m{\bf 1}_{\bar 1}\!\!\!&=\!\!\!&\frac12(G_m^-G_0^+{\bf 1}_{\bar 1}+G_0^+G_m^-{\bf 1}_{\bar 1}+mH_m{\bf 1}_{\bar 1})\\
    \!\!\!&=\!\!\!&\frac12( \frac{2}{\alpha}G_m^-L_0{\bf 1}_{\bar 0}+m\lambda^mH_0{\bf 1}_{\bar 1})\\
    \!\!\!&=\!\!\!& \frac12( \frac{2}{\alpha}(L_0+m)G_m^-{\bf 1}_{\bar 0}+m\lambda^mH_0{\bf 1}_{\bar 1})\\
     \!\!\!&=\!\!\!&\lambda^{m}(L_0+\frac12 mH_0+m){\bf 1}_{\bar 1}.
\end{eqnarray*}
We complete the proof.
\end{proof}

\subsection{The classification and structure of free $U(\mathfrak{h})$-modules of rank $2$ over $\R$}

We are now in the position to present the main result of this
section, which gives a complete classification of free
$U(\hh)$-modules of rank $2$ over the  Ramond $N=2$ superconformal
algebra.
\begin{thm}\label{thm-3}
\label{thm} Let $M$ be  an $\R$-module such that the restriction
of $M$ as a $\C[L_0,H_0]$-module is free of rank $2$. Then up to a
parity, $M\cong \Omega(\lambda,\alpha)$ for some
$\lambda,\alpha\in\C^*$ with the $\R$-module structure defined as
in (\ref{module1})-(\ref{C action}).
\end{thm}

\begin{proof}
The assertion follows directly from Lemma \ref{relations of R}, Corollary \ref{triviality}, Lemma \ref{lem2 for R}, Lemma \ref{lem4 for
 R},  and Lemma \ref{lem5 for
 R}.
\end{proof}

\begin{rem}
It follows from (\ref{iso between NS and R}) and (\ref{iso between T and R}) that Theorem \ref{thm-3} also gives a complete classification of free
$U(\hh)$-modules of rank $2$ over the Neveu-Schwarz and the topological $N=2$ superconformal algebras.
\end{rem}

The following result determines the isomorphism classes of the free $U(\hh)$-modules of rank 2 over the Ramond $N=2$ superconformal algebra $\R$.
\begin{thm}\label{iso class}
Let $\lambda,\mu, \alpha,\beta\in\C^*$. Then
$\Omega(\lambda,\alpha)\cong\Omega(\mu,\beta)$ as
$\R$-modules if and only if $\lambda=\mu$ and $\alpha=\beta$.
\end{thm}
\begin{proof}
The sufficiency is clear. Now suppose that
$\varphi:\Omega(\lambda,\alpha)\rightarrow\Omega(\mu,\beta)$
is an $\R$-module isomorphism. Let $1_{\bar 0}\in\C[x,y]$ and
$1_{\bar 1}\in\C[s,t]$ be the basis elements of the even part and the odd part as $U(\hh)$-modules. Assume
$\varphi(1_{\bar 0})=h(x,y)\in\C[x,y]$,  $\varphi(1_{\bar
1})=h'(s,t)\in\C[s,t]$. Then we have
\begin{eqnarray*}
\alpha h'(s,t)=\alpha\varphi(1_{\bar 1})=\varphi(\alpha 1_{\bar
1})=\varphi(G_0^-1_{\bar 0})=G_0^-\varphi(1_{\bar
0})=G_0^-h(x,y)=\beta h(s,t+1).
\end{eqnarray*}
This yields that
\begin{eqnarray*}\label{a and b}
h(s,t+1)=\frac{\alpha}{ \beta }h'(s,t).
\end{eqnarray*}
Hence,
\begin{equation*}
 \frac{2}{\alpha}x h(x,y)= G_0^+ h(s,t+1)= \frac{\alpha}{ \beta }G_0^+h'(s,t) =\frac{\alpha}{ \beta }G_0^+\varphi(1_{\bar
1})=\frac{\alpha}{ \beta }\varphi(G_0^+1_{\bar 1})=\frac{2}{ \beta }xh(x,y),
\end{equation*}
which implies $\alpha=\beta$.

Moreover, since
\begin{equation*}
\mu yh(x+1,y)=H_1\varphi(1_{\bar 0})=\varphi(H_11_{\bar 0})=\varphi(\lambda H_01_{\bar 0})=\lambda H_0\varphi(1_{\bar 0})=\lambda yh(x,y),
\end{equation*}
it follows that $\lambda=\mu$, as desired.
\end{proof}

We have the following theorem in which all submodules of $\Omega(\lambda,\alpha)$ are precisely determined for any $\lambda,\alpha\in\C^*$. In particular, $\Omega(\lambda,\alpha)$ is not simple.

\begin{thm}\label{thm-sub}
Let $\lambda,\alpha\in\C^*$. For any $ h(y)\in\C[y]$, let $M_h= h(y)\C[x,y]\oplus
h(t+1)\C[s,t]\subseteq \Omega(\lambda,\alpha)$ and $N_h= h(y)(x\C[x,y]+y\C[x,y])\oplus
h(t+1)\C[s,t]\subseteq \Omega(\lambda,\alpha)$. Then the following statements hold.
\begin{itemize}
\item[\rm(1)]
 The set $\{M_h, N_h \,|\, h(y)\in\C[y]\}$ exhausts all $\mathcal R$-submodules of $\Omega(\lambda,\alpha)$. Moreover, $N_h\subseteq M_h\subseteq M_{\tilde{h}}$ for any $h(y), \tilde{h}(y)\in\C[y]$ with $\tilde{h}(y)\mid h(y)$.
\item[\rm(2)]
 For any nonzero $h(y)\in\C[y]$, the quotient $\Omega(\lambda,\alpha)/M_h$ is a free $\C [L_0]$-module of rank $2\deg h(y)$.
\item[\rm(3)] Any maximal submodule of $\Omega(\lambda,\alpha)$ is of the form $M_h$ for some $h(y)\in\C[y]$ with $\deg h(y)=1$.
\end{itemize}
\end{thm}
\begin{proof}
(1) The second statement is obvious. It suffices to show the the first assertion. For that, suppose $\mathscr{M}=\mathscr{M}_{\bar
0}\oplus \mathscr{M}_{\bar 1}$ is a nonzero submodule of $\Omega(\lambda,\alpha)=\C[x,y]\oplus \C[s,t]$. Take a nonzero
$f(x,y)=\sum_{i=0}^{k}x^if_i(y)\in \C[x,y]$. We first prove the
following claim.

{\bf Claim:} $f(x,y)=\sum_{i=0}^{k}x^if_i(y)\in \mathscr{M}_{\bar
0}\Longleftrightarrow f_0(y),x f_i(y)\in \mathscr{M}_{\bar 0}$ for
$i=1,\cdots, k.$

The sufficient direction is obvious. For the necessary direction, take any $m\in\Z$, by (\ref{module1}),
we have
\begin{eqnarray*}
    \frac{1}{\lambda^m}L_mf(x,y)\!\!\!&=\!\!\!&(x+\frac12my)\sum_{i=0}^{k}(x+m)^if_i(y)\\
    \!\!\!&=\!\!\!&(x+\frac12my)\sum_{i=0}^{k}\sum_{j=0}^{k} \binom i j m^j  x^{i-j}f_i(y)\\
   \!\!\!&=\!\!\!&\sum_{i=0}^{k}\sum_{j=0}^{k} \binom i j m^j  x^{i-j+1}f_i(y)+\frac12 \sum_{i=0}^{k}\sum_{j=0}^{k} \binom i j m^{j+1}  x^{i-j}yf_i(y)\\
   \!\!\!&=\!\!\!&\sum_{j=0}^{k+1}\sum_{i=j-1}^{k} \binom i j m^j  x^{i-j+1}f_i(y)+\frac12 \sum_{j=0}^{k+1} \sum_{i=j-1}^{k}\binom {i} {j-1} m^{j}  x^{i-j+1}yf_i(y)\\
   \!\!\!&=\!\!\!&\sum_{j=0}^{k+1}a_jm^j\in M_{\bar 0},
\end{eqnarray*}
where
$$a_j=\sum_{i=j-1}^{k} \bigg(\binom i j +\frac12 \binom {i}
{j-1}y\bigg)  x^{i-j+1}f_i(y)\in\C[x,y],\, j=0,1,\cdots,k+1.$$ Taking
$m=1,\cdots, k+2$, we then obtain that $a_j\in M_{\bar 0}$ for
$j=0,1,\cdots,k+1$. Let $j=k,k+1$, we get
\begin{eqnarray}\label{ak}
     a_k=xf_k(y)+\frac12yf_{k-1}(y)+\frac12 kxyf_k(y)\in M_{\bar
    0},\ \ \ \
    a_{k+1}=\frac12 yf_k(y)\in M_{\bar
    0}.
\end{eqnarray}
Using (\ref{module2}), we have
\begin{eqnarray*}
    \frac{1}{\lambda^m}H_mf(x,y)= y\sum_{i=0}^{k}(x+m)^if_i(y)=y\sum_{i=0}^{k}\sum_{j=0}^{k} \binom i j m^j  x^{i-j}f_i(y)=\sum_{j=0}^{k}b_jm^j\in M_{\bar
    0}.
\end{eqnarray*}
Taking $m=1,\cdots, k+1$, we obtain that
$$b_j=\sum_{i=j}^{k}\binom i
j  x^{i-j}yf_i(y)\in M_{\bar 0}\,\,\text{ for }\,\,j=0,1,\cdots,k.$$ In
particular, we have
\begin{eqnarray}\label{bk}
    b_{k-1}=yf_{k-1}(y)+ kxyf_k(y)\in M_{\bar
    0}.
\end{eqnarray}
Using (\ref{ak}) and (\ref{bk}), one concludes that $xf_{k}(y)\in
M_{\bar 0}$. This implies $x^kf_{k}(y)\in M_{\bar 0}$. Thus
$$f(x,y)-x^kf_{k}(y)=\sum_{i=0}^{k-1}x^if_i(y)\in M_{\bar 0}.$$ In a
similar way, we have $$f_0(y),x f_i(y)\in M_{\bar 0}, \ \ \ \ i=1,\cdots, k.$$ Thus the Claim
holds.

Let $g(y),xh(y)\in \mathscr{M}_{\bar0}$ be nonzero polynomials such that
$\deg_y g(y),\deg_y xh(y) $ are minimal. Then for any $f(x,y)=\sum_{i=0}^{k}x^if_i(y)\in \mathscr{M}_{\bar
0}$, it follows from the Claim that $g(y)|f_0(y)$, $h(y)|f_i(y)$ for
$i=1,\cdots, k$.

Note that $xg(y)=L_0g(y)\in\mathscr{M}_{\bar0}$, we have $h(y)|g(y)$.
Since $H_0xh(y)=xyh(y)\in \mathscr{M}_{\bar0}$ and $H_1xh(y)=\lambda (x+1)yh(y)\in \mathscr{M}_{\bar0}$, we have $yh(y)\in \mathscr{M}_{\bar0}$. Thus $g(y)|yh(y)$. Therefore, $g(y)=c_1h(y)$ or $g(y)=c_2yh(y)$ for some nonzero $c_1, c_2\in\C$. We divide the following discussion into two cases.

Case (i): $g(y)=c_1h(y)$ for some $c_1\in\C^*$.

In this case, $\mathscr{M}_{\bar 0}$ is generated by $h(y)$, i.e., $\mathscr{M}_{\bar 0}=h(y)\C[x,y]$. For any
$u(s, t)\in\C[s,t]$, we have $h(t+1)u(s,t)=G_0^{-}(\frac{1}{\alpha}h(y)u(x,y-1))\in\mathscr{M}_{\bar 1}$.
 Hence, $h(t+1)\C[s,t]\subseteq\mathscr{M}_{\bar 1}$. On the other hand, take any $g(s,t)\in\mathscr{M}_{\bar 1}$, then
$G_0^+g(s,t)=\frac{2}{\alpha}xg(x,y-1)\in\mathscr{M}_{\bar 0}=h(y)\C[x,y]$. This implies that $g(s,t)\in h(t+1)\C[s,t]$,
i.e., $\mathscr{M}_{\bar 1}\subseteq h(t+1)\C[s,t]$. Consequently,  $\mathscr{M}_{\bar 1}=h(t+1)\C[s,t]$, and $\mathscr{M}=M_h$.

Case (ii): $g(y)=c_2yh(y)$ for some $c_2\in\C^*$.

In this case, $\mathscr{M}_{\bar 0}$ is generated by $xh(y)$ and $yh(y)$, i.e., $\mathscr{M}_{\bar 0}=h(y)(x\C[x,y]+y\C[x,y])$. For any
$u(s, t)\in\C[s,t]$, $$h(t+1)u(s,t)=G_1^{-}(\frac{1}{\lambda\alpha}h(y)xu(x-1,y-1))-G_0^{-}(\frac{1}{\alpha}h(y)xu(x,y-1))\in\mathscr{M}_{\bar 1}.$$
Hence, $h(t+1)\C[s,t]\subseteq\mathscr{M}_{\bar 1}$. On the other hand, take any $g(s,t)\in\mathscr{M}_{\bar 1}$, then
$$G_0^+g(s,t)=\frac{2}{\alpha}xg(x,y-1)\in\mathscr{M}_{\bar 0}=h(y)(x\C[x,y]+y\C[x,y]).$$ This implies that $g(s,t)\in h(t+1)\C[s,t]$,
i.e., $\mathscr{M}_{\bar 1}\subseteq h(t+1)\C[s,t]$. Consequently,  $\mathscr{M}_{\bar 1}=h(t+1)\C[s,t]$, and $\mathscr{M}=N_h$.

(2) is obvious.

(3) follows directly from the fundamental theorem of algebra and (1).
\end{proof}

By Theorem \ref{thm-sub}, any simple quotient of the $\mathcal R$-module $\Omega(\lambda,\alpha)$ is of the form $\Omega(\lambda,\alpha)/M_h$ for some $h(y)=y+a$ with $a\in\C$, which is denoted by $S(\lambda,\alpha, a)$. Then as a $\Z_2$-graded vector space, $S(\lambda,\alpha, a)=\C[x]\oplus\C[s]$ with
$S(\lambda,\alpha, a)_{\bar{0}}=\C[x]$ and $S(\lambda,\alpha, a)_{\bar{1}}=\C[s]$. It follows from Theorem \ref{thm-sub} and (\ref{module1})-(\ref{C action}) that
the $\mathcal R$-module structure of $S(\lambda,\alpha, a)$ is given as follows.
\begin{eqnarray}
    &&L_m f(x)=\lambda^m(x-\frac12
    ma)f(x+m),\quad  L_m g(s)=\lambda^m (s-\frac
    {1}{2}ma+\frac{1}{2}m)g(s+m),\label{qmodule1}\\
    &&H_m f(x)=-a\lambda^m f(x+m), \quad H_m g(s)=-(a+1)\lambda^m g(s+m),\label{qmodule2}\\
    &&G_m^+ f(x)=0 ,\quad G_m^+ g(s)=\lambda^m  \frac{2}{\alpha}(x-ma)g(x+m) , \label{qmodule3}\\
    &&G_m^{-} f(x)=\lambda^m \alpha f(s+m) , \quad G_m^{-} g(s)=0, \label{qmodule4}\\
    &&Cf(x)=Cg(s)=0,\label{qC action}
    \end{eqnarray}
where $f(x)\in\C[x], g(s)\in\C[s], m\in\Z$.

The following result determines the isomorphism classes of the simple $\mathcal R$-modules obtained in Theorem \ref{thm-sub}.

\begin{thm}\label{iso class of simples}
Let $\lambda,\mu,\alpha,\beta\in\C^*$, $a, b\in\C$. Then $S(\lambda, \alpha, a)\cong S(\mu, \beta, b)$ as $\mathcal R$-modules  if and only if
$\lambda=\mu$ and $a=b$.
\end{thm}

\begin{proof}
Let
$$S(\lambda,\alpha, a)=S(\lambda,\alpha, a)_{\bar{0}}\oplus S(\lambda,\alpha, a)_{\bar{1}}=\C[x]\oplus\C[s]=\C[x]1_{x}\oplus\C[s]1_{s},$$ and
$$S(\mu,\beta, b)=S(\mu,\beta, b)_{\bar{0}}\oplus S(\mu,\beta, b)_{\bar{1}}=\C[u]\oplus\C[v]=\C[u]1_{u}\oplus\C[v]1_{v}.$$

(1) Assume $\varphi:S(\lambda,\alpha, a)\rightarrow S(\mu,\beta, b)$ is an $\mathcal{R}$-isomorphism. Since for any $i\in\N$,
$$\varphi(x^i)=\varphi(L_0^i 1_x)=L_0^i \varphi(1_{x})=u^i\varphi(1_{x}),$$ it follows that $\varphi(S(\lambda,\alpha, a)_{\bar{0}})=\C[u]\varphi(1_x)$. This implies that $\varphi(1_{x})=c1_u$ for some $c\in\C^*$, and $\varphi(x^i)=cu^i$ for any $i\in\N$. Similar argument yields that $\varphi(1_{s})=d1_v$ for some $d\in\C^*$, and $\varphi(s^i)=dv^i$ for any $i\in\N$. Moreover, it follows from (\ref{qmodule1}) that
$$\lambda c(u-\frac{1}{2}a)1_u=\varphi(\lambda(x-\frac{1}{2}a))=\varphi(L_1 1_x)=L_1 \varphi(1_x)=\mu c(u-\frac{1}{2}b)1_u.$$
Consequently, $\lambda=\mu, a=b$.

(2) Suppose $\lambda=\mu, a=b$. Then we define the following linear mapping
\begin{eqnarray*}
\Phi:\,S(\lambda,\alpha, a)& \longrightarrow & S(\mu,\beta, b)\\
f(x)&\longmapsto& f(u),\\
g(s)&\longmapsto& \frac{\beta}{\alpha}g(v).
\end{eqnarray*}
It is readily seen that $\Phi$ is a $\mathcal{R}$-module isomorphism between $S(\lambda,\alpha, a)$ and $S(\mu,\beta, b)$. We complete the proof.
\end{proof}

\begin{lem}\label{simple iso lem}
Keep notations as in Theorem \ref{thm-sub}. Let $\lambda, \alpha\in\C^*$, $a\in\C$ and assume $h(y),\tilde{h}(y)\in\C[y]$ with $h(y)=(y+a)\tilde{h}(y)$. Then as a $\mathcal{R}$-module $M_{\tilde{h}}/M_h\cong S(\lambda,\alpha, a)$.
\end{lem}

\begin{proof}
As a $\Z_2$-graded vector space, $$M_{\tilde{h}}/M_h=(M_{\tilde{h}}/M_h)_{\bar{0}}\oplus(M_{\tilde{h}}/M_h)_{\bar{1}}
=\tilde{h}(y)\C[x]\oplus \tilde{h}(t+1)\C[s].$$
Define the following linear mapping
\begin{eqnarray*}
\Xi:\,S(\lambda,\alpha, a)& \longrightarrow & M_{\tilde{h}}/M_h\\
f(x)&\longmapsto& \tilde{h}(y)f(x),\\
g(s)&\longmapsto& \tilde{h}(t+1)g(s).
\end{eqnarray*}
By (\ref{module1})-(\ref{C action}) and (\ref{qmodule1})-(\ref{qC action}), it is a routine to check that $\Xi$ is a $\mathcal{R}$-module isomorphism, so that
$M_{\tilde{h}}/M_h\cong S(\lambda,\alpha, a)$.
\end{proof}

As a consequence, we can precisely determine the decomposition series of the quotient module $\Omega(\lambda,\alpha)/M_h$ for any $h(y)\in\C[y]$.

\begin{coro}
Keep notations as in Theorem \ref{thm-sub}. In particular, $\lambda, \alpha\in\C^*$. Let $h(y)$ be a monic polynomial of degree $n$. Suppose $a_1,\cdots, a_n$ are all roots of $h(y)$ in $\C$ (counting the multiplicity). Set $h_i(y)=(y-a_1)(y-a_2)\cdots (y-a_i)$ and $M_i=M_{h_i}/M_h$ for $i=1,\cdots, n$. Then
$$\Omega(\lambda,\alpha)/M_h\supset M_1\supset M_2\supset\cdots\supset M_{n-1}\supset M_n=0$$
is a decomposition series of the quotient module $\Omega(\lambda,\alpha)/M_h$. Moreover,
$$(\Omega(\lambda,\alpha)/M_h)/ M_1\cong S(\lambda,\alpha, -a_1),\, M_i/M_{i+1}\cong S(\lambda, \alpha, -a_{i+1}),\,1\leq i\leq n-1.$$
Consequently, in the Grothendieck group,
$$[\Omega(\lambda,\alpha)/M_h]=\sum\limits_{i=1}^n[S(\lambda, \alpha, -a_i)].$$
\end{coro}

\begin{proof}
The assertion follows directly from Lemma \ref{simple iso lem} and the following isomorphisms
$$(\Omega(\lambda,\alpha)/M_h)/(M_{h_1}/M_h)\cong \Omega(\lambda,\alpha)/M_{h_1}, (M_{h_i}/M_h)/(M_{h_{i+1}}/M_h)\cong M_{h_i}/M_{h_{i+1}}, 1\leq i\leq n-1.$$
\end{proof}

\subsection{Relationship between non-weight representations of $N=2$ and $N=1$ superconformal algebras}
In physics, $N = 1$ supersymmetric extension of conformal symmetry, called superconformal symmetry, promotes string to superstring. The
extension of string theory to supersymmetric string theories identifies $N = 1$ superconformal algebra as the symmetry algebras of closed superstrings. The $N =1$ superconformal algebras include two types, that is, the Ramond type and the Neveu-Schwarz type (cf. \cite{Ne, Ra}). They are two possible superextensions of the Virasoro algebra for the case of one fermionic current. Let us first recall the definitions of the (centerless) $N=1$ superconformal algebras of Ramond type and Neveu-Schwarz type.
\begin{de}
Let $\epsilon =0\, \mbox{or} \, \frac{1}{2}$. the (centerless) $N=1$ superconformal algebra is an infinite dimensional Lie
superalgebra whose even part is spanned by $\{\L_{n} \mid n \in \Z\}$ and odd part is spanned by $\{\G_{r} \mid r \in \epsilon +\Z \}$ subject to the following relations
\begin{eqnarray*}
    &&[\L_{m}, \L_{n}]=(m-n)\L_{m+n},\\
    &&[\L_m,\G_r]=(\frac{1}{2}m-r)\G_{m+r},\\
    &&[\G_r,\G_s]=2\L_{r+s}, \ \ \ \ \ \ m,n\in\Z, r,s\in \epsilon+\Z.
\end{eqnarray*}
When $\epsilon=0$, it is called the Ramond algebra, denote by $\mathscr{R}$. When $\epsilon=\frac{1}{2}$, it is called the Neveu-Schwarz algebra, denoted
by $\mathscr{N}\hspace{-1.5mm}\mathscr{S}$.
\end{de}

Define the following two linear mappings
\begin{eqnarray*}
\Upsilon_1:\,\mathscr{N}\hspace{-1.5mm}\mathscr{S}& \longrightarrow &\mathscr{R}\\
\L_m&\longmapsto& \frac{1}{2} \L_{2m},\,\,\forall\,m\in\Z,\\
\G_{r}&\longmapsto& \frac{1}{\sqrt{2}} \G_{2r},\,\,\forall\,r\in\frac12+\Z
\end{eqnarray*}
and
\begin{eqnarray*}
\Upsilon_2:\,\mathscr{R}& \longrightarrow &\mathcal{R}/\C C\\
\L_m&\longmapsto& L_{m},\\
\G_m&\longmapsto& \frac{1}{\sqrt{2}}(G_{m}^++G_m^{-}),\,\,\forall\,m\in\Z.
\end{eqnarray*}
It is straightforward to verify that both $\Upsilon_1$ and $\Upsilon_2$ are injective Lie superalgebra homomorphisms. Then $\mathscr{N}\hspace{-1.5mm}\mathscr{S}$ can be regarded as a subalgebra of $\mathscr{R}$, and $\mathscr{R}$ can be regarded as a subalgebra of 
$\mathcal{R}/\C C$. Hence, any $\mathcal{R}$-module with trivial $C$-action is both a $\mathscr{R}$-module and an $\mathscr{N}\hspace{-1.5mm}\mathscr{S}$-module.

In the following we will exploit the relationship between these simple $\mathcal R$-modules
$S(\lambda,\alpha, a)$ obtained in Theorem \ref{thm-sub} and those non-weight modules over $N=1$ superconformal algebras when restricted to the Cartan subalgebra are free of rank one or two classified in \cite{YYX}.
Let us first briefly introduce them.  For
$\lambda\in\C^*, c\in\C$, let
$$\Omega_{\mathscr{R}}(\lambda, c)=\Omega_{\mathscr{R}}(\lambda, c)_{\bar{0}}\oplus\Omega_{\mathscr{R}}(\lambda, c)_{\bar{1}}=\C[x^2]\oplus x\C[x^2]$$
which is a $\mathscr{R}$-module with the module structure defined as in \cite[Proposition 2.4]{YYX}, and
$$\Omega_{\mathscr{N}\hspace{-1.5mm}\mathscr{S}}(\lambda, c)=\Omega_{\mathscr{N}\hspace{-1.5mm}\mathscr{S}}(\lambda, c)_{\bar{0}}\oplus
\Omega_{\mathscr{N}\hspace{-1.5mm}\mathscr{S}}(\lambda, c)_{\bar{1}}=\C[x]\oplus\C[y]$$
which is an $\mathscr{N}\hspace{-1.5mm}\mathscr{S}$-module with the module structure defined as in  \cite[Proposition 2.8]{YYX}. These modules $\Omega_{\mathscr{R}}(\lambda, c)$ over the Ramond algebra are free of rank 1 when restricted as modules over its Cartan subalgebra $\C\L_0\oplus\C\G_0$, while those modules $\Omega_{\mathscr{N}\hspace{-1.5mm}\mathscr{S}}(\lambda, c)$ over the Neveu-Schwarz algebra are free of rank 2 when restricted as modules over its Cartan subalgebra $\C\L_0$.

\begin{prop}\label{relation}
Let $\lambda, \alpha\in\C^*, a\in\C$. Then the following statements hold.
\begin{itemize}
\item[(1)] As a $\mathscr{R}$-module, $S(\lambda, \alpha, a)\cong\Omega_{\mathscr{R}}(\lambda, -\frac{1}{2}a)$.
\item[(2)] As an $\mathscr{N}\hspace{-1.5mm}\mathscr{S}$-module, $S(\lambda, \alpha, a)\cong\Omega_{\mathscr{N}\hspace{-1.5mm}\mathscr{S}}(\lambda^2, -\frac{1}{2}a)$.
\end{itemize}
\end{prop}

\begin{proof}
(1) Define the following linear mapping
\begin{eqnarray*}
\Psi:\, \Omega_{\mathscr{R}}(\lambda, -\frac{1}{2}a) &\longrightarrow &S(\lambda,\alpha, a)=\C[x]\oplus\C[s]\\
f(x^2)&\longmapsto & f(x),\\
xg(x^2)&\longmapsto & \frac{\alpha}{\sqrt{2}}g(s).
\end{eqnarray*}
With aid of (\ref{qmodule1}), (\ref{qmodule3}), (\ref{qmodule4}), and \cite[(2.2)-(2.5)]{YYX}, it is easy to verify that $\Psi$ is a $\mathscr{R}$-module isomorphism of $\Omega_{\mathscr{R}}(\lambda, -\frac{1}{2}a)$ and $S(\lambda,\alpha, a)$.

(2) follows from (1) and \cite[Proposition 2.11]{YYX}.
\end{proof}

We then have the following consequence which gives a sufficient and necessary condition for a simple $\mathcal{R}$-module $S(\lambda, \alpha, a)$
to be simple as a $\mathscr{R}$-module and an $\mathscr{N}\hspace{-1.5mm}\mathscr{S}$-module.

\begin{coro}
Let $\lambda, \alpha\in\C^*, a\in\C$. Then we have
\begin{itemize}
\item[(1)] $S(\lambda, \alpha, a)$ is simple as a $\mathscr{R}$-module if and only if $a\neq 0$.
\item[(2)]  $S(\lambda, \alpha, a)$ is simple as an $\mathscr{N}\hspace{-1.5mm}\mathscr{S}$-module if and only if $a\neq 0$.
\end{itemize}
\end{coro}

\begin{proof}
It follows directly from Proposition \ref{relation} and \cite[Theorem 2.6, Corollary 2.12]{YYX}.
\end{proof}

\subsection*{Acknowledgement} The authors would like to express their sincere gratitude to the referee for helpful comments and suggestion which improve the original manuscript.

\end{document}